\documentclass{amsart}
\usepackage[utf8]{inputenc}
\usepackage[english]{babel}

\usepackage[dvips]{graphics,epsfig}
\usepackage{times}
\usepackage{amsmath,amsthm,amssymb}
\usepackage{array}
\usepackage{bbm}
\usepackage{xcolor}
\usepackage{verbatim}
\usepackage{upgreek}
\usepackage[foot]{amsaddr}

\usepackage{geometry}

\newtheorem{theorem}{Theorem}[section]
\newtheorem{thm}[theorem]{Theorem}
\newtheorem{lemma}[theorem]{Lemma}
\newtheorem{question}[theorem]{Question}
\newtheorem{cor}[theorem]{Corollary}

\newtheorem{fact}[theorem]{Fact}
\newtheorem{claim}[theorem]{Claim}
\newtheorem{subclaim}[theorem]{Subclaim}

\theoremstyle{definition}
\newtheorem{definition}[theorem]{Definition}

\theoremstyle{remark}
\newtheorem{remark}[theorem]{Remark}

\def\Ind{\setbox0=\hbox{$x$}\kern\wd0\hbox to 0pt{\hss$\mid$\hss} \lower.9\ht0\hbox to 0pt{\hss$\smile$\hss}\kern\wd0}
\def\Notind{\setbox0=\hbox{$x$}\kern\wd0\hbox to 0pt{\mathchardef \nn=12854\hss$\nn$\kern1.4\wd0\hss}\hbox to 0pt{\hss$\mid$\hss}\lower.9\ht0 \hbox to 0pt{\hss$\smile$\hss}\kern\wd0}
\def\ind{\mathop{\mathpalette\Ind{}}}
\def\nind{\mathop{\mathpalette\Notind{}}}
\numberwithin{equation}{section}

\title{An $\mathrm{NSOP}_{1}$ theory without the existence axiom}
\author{Scott Mutchnik
}
\address{Department of Mathematics, Statistics, and Computer Science, University of Illinois at Chicago}
\email{mutchnik@uic.edu}
\thanks{This work was supported by the NSF under Grant No. DMS-2303034.}

\begin{document}

\begin{abstract}
    Answering a question of Dobrowolski, Kim and Ramsey, we find an $\mathrm{NSOP}_{1}$ theory that does not satisfy the existence axiom. 
\end{abstract}

\maketitle

\section{Introduction}

One of the core informal questions of model theory is to determine what role stability theory, famously introduced by Morley (\cite{M65}) and Shelah (\cite{Sh90}) to classify the number of non-isomorphic models (of a given size) of a first-order theory, can play in describing theories that are themselves unstable--and that may not even be simple. This project was initiated in large part in celebrated work of Kim (\cite{Kim98}) and Kim and Pillay (\cite{KP99}), who showed that the (forking-)independence relation has many of the same properties in simple theories that it does in stable theories, and in fact that these properties characterize simplicity. In a key step towards the non-simple case, Kaplan and Ramsey (\cite{KR17}) then defined \textit{Kim-independence} over models, which generalizes the definition of forking-independence:

\begin{definition}  \label{kimindependence}
Let $M \models T$. A formula $\varphi(x, b)$ \emph{Kim-divides} over $M$ if there is an $M$-invariant Morley sequence $\{b_{i}\}_{i \in \omega}$ starting with $b$ such that $\{\varphi(x, b_{i})\}_{i \in \omega}$ is inconsistent. A formula  $\varphi(x, b)$ \emph{Kim-forks} over $M$ if it implies a (finite) disjunction of formulas Kim-dividing over $M$. We write $a \ind^{K}_{M} b$, and say that $a$ is \emph{Kim-independent} from $b$ over $M$, if $\mathrm{tp}(a/Mb)$ does not include any formulas Kim-forking over $M$.
\end{definition}

In, for example, \cite{CR15}, \cite{KR17}, \cite{KR19}, \cite{KRS19}, it is shown that Kim-independence has many of the same properties in theories without the first strict order property--$\mathrm{NSOP}_{1}$ theories--that forking-independence has in simple theories. It is also shown that $\mathrm{NSOP}_{1}$ theories have a characterization in terms of properties of Kim-independence, similarly to how simplicity has a characterization in terms of properties of forking-independence. However, in contrast to the case of simple theories, this work only describes the properties of a relation $a \ind_{M}^{K} b$ defined when $M$ is a \textit{model}, leaving open the question of whether independence phenomena in $\mathrm{NSOP}_{1}$ theories extend from independence over models to independence over \textit{sets}.

The following axiom of first-order theories, though introduced earlier under different names (see, e.g. \cite{CK09}), was defined by Dobrowolski, Kim and Ramsey in \cite{DKR22}:

\begin{definition}
        A theory $T$ satisfies the \textit{existence axiom} if no type $p \in S(A)$ forks over $A$.
\end{definition}

This is equivalent to every type $p \in S(A)$ having a global extension that does not fork over $A$. All simple theories satisfy the existence axiom (\cite{Kim98}), while the circular ordering is an example of a dependent theory not satisfying the existence axiom (\cite{Kimthesis96}, Example 2.11). In \cite{DKR22} it is shown that, in $\mathrm{NSOP}_{1}$ theories satisfying the existence axiom, it is possible to extend the definition of Kim-independence $a \ind_{C}^{K} b$ from the case where $C$ is a model to the case where $C$ is an arbitrary set, so that the relation $a \ind_{C}^{K} b$ will have similar properties to Kim-independence over models in general $\mathrm{NSOP}_{1}$ theories. Specifically, because, under the existence axiom, (nonforking-)Morley sequences are defined over arbitrary sets, they can be used in place of invariant Morley sequences\footnote{It is easy to see that invariant Morley sequences are not defined over arbitrary sets in $\mathrm{NSOP}_{1}$ theories: for example, any set $A$ such that $\mathrm{acl}(A) \neq \mathrm{dcl}(A)$.} to define Kim-dividing of a formula over a set:

\begin{definition} \label{setkimindependence} (\cite{DKR22})
Let $T$ be a $\mathrm{NSOP}_{1}$ theory, and let $C \subset \mathbb{M}$ be an arbitrary set. A formula $\varphi(x, b)$ \emph{Kim-divides} over $C$ if there is a (nonforking-)Morley sequence $\{b_{i}\}_{i \in \omega}$ over $C$ starting with $b$ such that $\{\varphi(x, b_{i})\}_{i \in \omega}$ is inconsistent. A formula  $\varphi(x, b)$ \emph{Kim-forks} over $C$ if it implies a (finite) disjunction of formulas Kim-dividing over $C$. We write $a \ind^{K}_{A} b$, and say that $a$ is \emph{Kim-independent} from $b$ over $C$, if $\mathrm{tp}(a/Ab)$ does not include any formulas Kim-forking over $C$.
\end{definition}

Dobrowolski, Kim and Ramsey (\cite{DKR22}) show that, in an $\mathrm{NSOP}_{1}$ theory \textit{satisfying the existence axiom}, the Kim-independence relation $\ind^{K}$ as defined above over arbitrary sets satisfies symmetry and the independence theorem (for Lascar strong types), and that Kim-dividing over arbitrary sets satisfies Kim's lemma and coincides with Kim-forking. Chernikov, Kim and Ramsey, in \cite{CKR20}, show even more properties of Kim-independence over arbitrary sets in $\mathrm{NSOP}_{1}$ theories satisfying the existence axiom, including transitivity and witnessing. (All of these properties correspond to the properties of Kim-independence over models proven in general $\mathrm{NSOP}_{1}$ theories in \cite{KR17}, \cite{KR19}.) Motivated by these results, \cite{DKR22}, and later \cite{KKL24}, ask:

\begin{question}\label{mainquestion}
   Does every $\mathrm{NSOP}_{1}$ theory satisfy the existence axiom?
\end{question}

We show that the answer to this question is no:

\begin{thm}
\label{main}  There is an $\mathrm{NSOP}_{1}$ theory not satisfying the existence axiom.
\end{thm}

This contrasts with the work of Kim, Kim and Lee in \cite{KKL24}, where, for the definition of Kim-forking over sets given by Dobrowolski, Kim and Ramsey in \cite{DKR22} (Defintion \ref{setkimindependence} above), it is shown that no type $p \in S(A)$ Kim-forks over $A$ (i.e. contains a formula Kim-forking over $A$), where $A$ is an arbitrary set in an $\mathrm{NSOP}_{1}$ theory. In fact, our example is the first known example of a theory without the strict order property, or $\mathrm{NSOP}$ theory, not satisfying the existence axiom. Prior to our results, the theory of an algebraically closed field with a generic multiplicative endomorphism, constructed by d'Elbée in \cite{D23}, was suggested there as a candidate for an $\mathrm{NSOP}_{1}$ theory without the existence axiom; whether the theory constructed by d'Elbée  satisfies the existence axiom was left unresolved in that article. However, as discussed in a personal communication with d'Elbée (\cite{Dbanff2023}), it is expected that that theory, which was shown in \cite{D23} to be $\mathrm{NSOP}_{1}$, actually does satisfy the existence axiom.

Our construction is based on the theory of $\omega$-stable free pseudoplanes, a classical example of a non-one-based $\omega$-stable theory with trivial forking discussed in, say, \cite{P96}. The theory of $\omega$-stable free pseudoplanes is the theory of undirected graphs of infinite minimum degree without cycles. Variations of this construction appear in e.g. \cite{HP21}, \cite{AT23}, \cite{T14}, \cite{KoponenConjecture}. Note also that some arguments from the below, including the strategy, in axiom schema $T_{4}$, of requiring that $n$ connected sets of pairwise distance greater than $2^{n}$ can be colored independently, and the associated Claims \ref{coloringextension} and \ref{coloringinterpolation}, their proofs, and their application in proving claim (***), are formally similar to those of Section 3 of Chernikov, Hrushovski, Kruckman, Krupiński, Moconja, Pillay and Ramsey (\cite{CHKKMPR23}).

\section{The construction}

Let $\mathcal{L}$ be the language with sorts $P$ and $O$, symbols $R_{1}$ and $R_{2}$ for binary relations on $O$, and symbols $\rho_{1}$ and $\rho_{2}$ for binary relations between $P$ and $O$. Call an $\mathcal{L}$-structure $A$ \textit{copacetic} if:

(C1) For $i = 1, 2$, $R_{i}(A)$ is a symmetric, irreflexive relation on $O(A)$, and the two are mutually exclusive: for $a_{1}, a_{2} \in O(A)$, $A \not\models R_{1}(a_{1}, a_{2}) 
 \wedge R_{2}(a_{1}, a_{2})$.

 (C2) The relation $R_{1}(A) \cup R_{2}(A)$ has no loops on $O(A)$ (i.e. there are no distinct $a_{0} \ldots a_{n-1} \in O(A)$, $n > 2$, and $i_{1} \ldots i_{n} \in \{1, 2\}$ so that, for $0 \leq j \leq n-1$, $A \models R_{i_{j}}(a_{i}, a_{i+1 \mathrm{\: mod \:} n})$).

 (C3) For all $b \in P(A)$, $a \in O(A)$, exactly one of $A \models \rho_{1}(b, a)$ and $A \models \rho_{2}(b, a)$ hold.

 (C4):
 
(a) For each $b \in P(A)$, there are no distinct $a_{1}, a_{2} \in O(A)$, so that there there is some $a_{*} \in O(A)$ so that $A \models R_{1}(a_{1}, a_{*}) \wedge R_{1}(a_{2}, a_{*})$, and $A \models \rho_{1}(b, a_{1}) \wedge \rho_{1}(b, a_{2}) $.

(b) For each $b \in P(A)$, there are no distinct $a_{1}, a_{2}, a_{3} \in O(A)$, so that there there is some $a_{*} \in O(A)$ so that $A \models R_{2}(a_{1}, a_{*}) \wedge R_{2}(a_{2}, a_{*}) \wedge R_{2}(a_{3}, a_{*})$, and $A \models \rho_{2}(b, a_{1}) \wedge \rho_{2}(b, a_{2}) \wedge \rho_{2}(b, a_{3})$.

Let $A$ be copacetic, and let $b \in P(A)$, $a \in O(A)$, $i \in {1, 2}$. Then there are at most $i$ many $a' \in O(A)$ with $A \models R_{i}(a, a')$ and $A \models \rho_{i}(b, a') $. For $N$ the number of such $a'$ that are defined, let $b^{A, j}_{\to^{i} }(a)$, $1 \leq j \leq N$, denote the $N$ many such $a'$ (so if $i = 2$ and there are two such $a'$, make an arbitrary choice of which is $b^{A,1 }_{\to^{i} }(a)$ and which is $b^{A,2 }_{\to^{i} }(a)$, while if $i = 1$, then $b^{A,1 }_{\to^{i} }(a)$ is the sole such $a'$, if it exists.) 

If $B$ is copacetic and $A \subset B$ is a substructure of $B$ (and is therefore also copacetic), call $A$ \textit{closed} in $B$ (denoted $A \leq B$) if 

(i) For  $b \in P(A)$, $a \in O(A)$, $i \in \{1, 2\}$, $1 \leq j \leq i$, if $b^{B, j}_{\to^{i} }(a)$ exists, then $b^{B, j}_{\to^{i} }(a) \in A$. 

(ii) Any $R_{1}(B) \cup R_{2}(B)$-path between nodes of $O(A)$ lies in $O(A)$: for all $a_{1}, a_{n} \in O(A)$, $a_{2}, \ldots a_{n-1} \in O(B)$ which are distinct and distinct from $a_{1}$, $a_{n}$, and $i_{1} , \ldots, i_{n-1} \in \{1, 2\}$, if $A \models R_{i_{j}}(a_{i}, a_{i+1})$ for $1 \leq j \leq n-1$,  then for all $2 \leq i \leq n-1$, $a_{i} \in O(A)$.

Call a copacetic $\mathcal{L}$-structure $A$ \textit{connected} if $R_{1}(A) \cup R_{2}(A)$ forms a connected graph on $O(A)$: for all $a, a' \in O(A)$, there are  $a_{2}, \ldots a_{n-1} \in O(A)$, $i_{1} , \ldots, i_{n-1} \in \{1, 2\}$ for some $n$, so that for $a_{1} = a$, $a_{n} = a'$,  $A \models R_{i_{j}}(a_{i}, a_{i+1})$ for all $1 \leq j \leq n-1$. So if $B$ is copacetic and $A \subseteq B$, connectedness of $A$ supplants requirement (ii) of $A$ being closed in $B$. Call a subset of $A$ a \textit{connected component} of $A$ if it is a maximal connected subset of $O(A)$.

Let $O$ be an undirected graph without cycles and with a $2$-coloring of its edges, with $R_{1}$, $R_{2}$ denoting edges of either color. Let $\rho_{1}, \rho_{2} \subset O$, $O = \rho_{1} \cup \rho_{2}$, $\rho_{1} \cap \rho_{2} = \emptyset$  be a coloring of the vertices of $O$ so that, for $i = \{1, 2\}$, no $i+1$ distinct vertices of $O$, lying on the boundary of the same $R_{i}$-ball of radius $1$ (i.e. they have a common $R_{i}$-neighbor), are both colored by $\rho_{i}$. Then we call $\rho_{1}, \rho_{2}$ a \textit{(C4)-coloring} of $O$. For $A$ copacetic, $O' \subseteq O(A)$, and $\rho_{1}, \rho_{2} \subseteq O'$ a (C4)-coloring of $O'$, say that $b \in P(A)$ \textit{induces} the (C4)-coloring $\rho_{1}, \rho_{2}$ on $O'$ if for $ i \in \{1, 2\}$, $\rho_{i} = \{a \in O': A \models \rho_{i}(b, a)\}$.

The following assumptions on an $\mathcal{L}$-structure $A$ are expressible by a set of first-order sentences:

($T_{1}$) (Copaceticity) $A$ is copacetic.

($T_{2}$) (Completeness) For $b \in P(A)$, $a \in O(A)$, $i \in \{1, 2\}$, if $A \models \neg \rho_{i}(b, a) $, then $b^{A, j}_{\to^{i} }(a)$ exists for $1 \leq j \leq i$. 

($T_{3}$) (Tree extension) For $C \subseteq A$, and $B \geq C$ finite and copacetic with $P(B) = P(C)$, there is an embedding $\iota: B \hookrightarrow  A$ with $\iota|_{C} = \mathrm{id}_{C}$.

($T_{4}$) (Parameter introduction) For any $n < \omega$, finite connected sets $O_{1}, \ldots, O_{n} \subseteq O(A)$ so that there does not exist an $R_{1} \vee R_{2}$-path of length at most $2^{n}$ between a vertex of $O_{i}$ and a vertex of $O_{j}$ for any distinct $i, j \leq n$ (so in particular, $O_{i}$ and $O_{j}$ are disjoint), and (C4)-colorings $\rho^{i}_{1}, \rho^{i}_{2} \subseteq O_{i}$ of $O_{i}$ for $i \leq n$, there are infinitely many $b \in P(A)$ so that, for each $i 
\leq n$, $b$ induces the (C4)-coloring  $\rho^{i}_{1}, \rho^{i}_{2} \subseteq O_{i}$ on $O_{i}$.

Let $T^{\not\exists}=T_{1} \cup T_{2} \cup T_{3} \cup T_{4}$. We claim that $T^{\not\exists}$ is consistent. This will follow by induction from the following three claims:

(*) If $A$ is copacetic, $b \in P(A)$, $a \in O(A)$, $i \in \{1, 2\}$, and $1 \leq j \leq i$, there is a copacetic $\mathcal{L}$-structure $A' \supseteq A$ (i.e. containing $A$ as a substructure) so that $b^{A', j}_{\to^{i} }(a)$ exists.

(**) If $A$ is copacetic, $C \subseteq A$, and $B \geq C$ is finite and copacetic with $P(B) = P(C)$, there is a copacetic structure $A' \supseteq A$ and an embedding $\iota: B \hookrightarrow  A'$ with $\iota|_{C} = \mathrm{id}_{C}$. 

(***) If $A$ is copacetic, for any $n < \omega$, finite connected sets $O_{1}, \ldots, O_{n} \subseteq O(A)$ so that there does not exist an $R_{1} \vee R_{2}$-path of length at most $2^{n}$ between a vertex of $O_{i}$ and a vertex of $O_{j}$ for any distinct $i, j \leq n$, and (C4)-colorings $\rho^{i}_{1}, \rho^{i}_{2} \subseteq O_{i}$ of $O_{i}$ for $i \leq n$, there is some copacetic $A' \supset A$ and $p \in P(A'\backslash A)$ so that, for each $i 
\leq n$, $p$ induces the (C4)-coloring  $\rho^{i}_{1}, \rho^{i}_{2} $ on $O_{i}$.

We first show (*). Suppose $b^{A, j}_{\to^{i} }(a)$ does not already exist. Let $A' = A \cup\{*\}$ for $*$ a new point of sort $O$, and let $R_{i}(A') = R(A) \cup \{(a, *), (*, a)\}$, $R_{3-i}(A')=R_{3-i}(A)$, $\rho_{3-i}(A') = \rho_{3-i}(A) \cup (P(A) \backslash \{b\}) \times \{*\}$, $\rho_{i}(A') = \rho_{i}(A) \cup \{(b, *)\}$. Then $A'$ is copacetic; first, (C1)-(C3) are clearly satisfied. Second, no point of $P(A) \backslash \{b\}$ can witness a failure of (C4), because $A$ is copacetic and $*$ is not an $R_{3-i}$-neighbor of any point of $O(A')$. Finally, neither can $b$ witness a failure of (C4), because $*$ is not an $R_{i-3}$ neighbor of any point of $O(A')$ and $a$ is the unique $R_{i}$-neighbor of $*$ in $O(A')$, that failure must by copaceticity of $A$ be witnessed on the boundary of the $R_{i}$-ball of radius $1$ centered at $a$. But because $b^{A, j}_{\to^{i} }(a)$ does not already exist, there are fewer than $i$ many $R_{i}$-neighbors $a'$ of $a$ in $O(A)$ with $A \models \rho_{i}(b, a')$, so there are at most $i$ many $R_{i}$-neighbors $a'$ of $A$ in $O(A')$ with $A \models \rho_{i}(b, a')$; thus the failure of (C4) is not in fact witnessed on this ball's boundary. By construction, we can then choose $b^{A', j}_{\to^{i} }(a)= *$.

We next show (**). For this we need the following claim:

\begin{claim}
\label{coloringextension}    Let $O'$ be undirected graph without cycles and with a $2$-coloring of its edges, with $R_{1}$, $R_{2}$ denoting edges of either color. Let $O \subset O'$ be a subgraph with the induced coloring, and with $O \leq O'$ (in the sense of (ii), so any path between two vertices of $O'$ consisting of edges of any colors is contained in $O$). Then any (C4)-coloring $\rho_{1}, \rho_{2}$ of $O$ extends to a (C4)-coloring $\rho_{1}', \rho'_{2}$ of $O'$, which has the following additional property: if $v \in O' \backslash O$ has an $R_{i}$-neighbor in $O$, then $v \in \rho'_{3-i}$.
\end{claim}

\begin{proof}
By the assumption on paths in $O'$ between two vertices of $O$, we may decompose $O' \backslash O$ into a disjoint union $\sqcup O^{i}$ of connected subgraphs, so that each $O^{i}$ has at most one vertex $v_{i}$ with any neighbors in $O$; $v_{i}$ will in fact have only one neighbor $w_{i}$ in $O$. For every $O^{i}$ all of whose vertices have no neighbors in $O$, let $v_{i}$ be an arbitrary vertex of $O^{i}$. Then inductively, we can order each $O^{i}$ as a tree (i.e., a partial order with linearly ordered downsets) so that $v_{i}$ is the root, any node's immediate successors are all neighbors of that node and, among any two neighbors, one must be an immediate successor of the other, and each maximal linearly ordered set is well-ordered of order type at most $\omega$. Extend $\rho_{1}, \rho_{2}$ on each $O_{i}$ as follows, starting from $v_{i}$ and proceeding by induction. If $w_{i}$ is an $R_{i}$-neighbor of $v_{i}$, color $v_{i}$ by $\rho'_{3-i}$; otherwise, color $v_{i}$ arbitrarily. Then for each vertex $v$ of $O^{i}$, each immediate successor of $v$ will be an $R_{i}$-neighbor of $v$ for some $i \in \{1, 2\}$; color it by $\rho_{3 - i}$. No two distinct vertices $v, w$ with an $R_{i}$-distance of exactly $2$ in $O'$ can then be colored in $\rho^{i}$: if the $R_{i}$-path between $v$ and $w$ goes through $O$, then whichever of $v$ and $w$ is not in $O$ will be colored by $\rho'_{3-i}$, and if the $R_{i}$-path between $v$ and $w$ stays in $O^{i}$, whichever of $v$ and $w$ is not the least in the path will be colored by $\rho'_{3-i}$. This shows that $\rho'_{1}, \rho'_{2}$ is a (C4)-coloring, and the additional property is immediate from the construction.

\end{proof}

Now, take the set $A'$ to be the disjoint union of $A$ and $B$ over $C$; for notational simplicity we identify $A$, $B$ and $C$ respectively with their images in $A'$. Let the $R_{1}, R_{2}$-structure on $O(A')$ be given as follows: the identification on $O(A)$ and $O(B)$ preserves the $R_{i}$-structure, and $O(A)$ and $O(B)$ are freely amalgamated over $O(C)$ (i.e. there are no $R_{i}$-edges between $O(A\backslash C)$ and $O(B\backslash C)$). This guarantees (C1), and also guarantees (C2) by condition (ii) of $C \leq B$. Now let the $\rho_{1}, \rho_{2}$-structure on $A'$ be given as follows: the identifications on $A$ and $B$ preserve the $\rho_{i}$-structure, which, by the assumption that $P(B) = P(C)$, leaves us only by way of satisfying (C3) to define the $\rho_{1}, \rho_{2}$-structure on $P(A\backslash C) \times O(B \backslash C)$, and this will be the following. Let $p \in P(A\backslash C)$; the requirement that the $\rho_{i}$-structure on $A$ is preserved tells us the (C4)-coloring that $p$ induces on $O(C)$, and we extend this to a (C4)-coloring on $O(B)$ as in Claim \ref{coloringextension}; here we use condition (ii) of $C\leq B$. We have defined the full $\mathcal{L}$-structure on $A'$, so it remains to show (C4); in other words, we show that $p$ induces a (C4)-coloring on $A'$ in the case where $p \in P(C)$ and in the case where $p \in P(A\backslash C)$. In either case $p$ induces a (C4)-coloring on $A$ and on $B$, so the only way (C4) can fail is on the boundary of an $R_{i}$-ball of radius $1$ centered at $c \in C$. But (C4) cannot fail on the boundary of this ball, because it does not fail there in $A$, and by condition (i) of $C \leq B$ in the first case, or by the additional clause of Claim \ref{coloringextension} in the second case, any $R_{i}$-neighbor of $c$ in $B\backslash C$ is colored by $\rho_{3-i}$ in the (C4)-coloring induced by $p$ on $B$, so (C4) still cannot fail on the boundary of this ball in $A'$. So $A'$ is a copacetic $\mathcal{L}$-structure containing $A$, and clearly, the embedding $\iota$ as in (**) exists, so this shows (**).

Finally, we show (***). For this, we need an additional combinatorial claim:

\begin{claim}
   \label{coloringinterpolation} Let $O$ be be undirected graph without cycles and with a $2$-coloring of its edges, with $R_{1}$, $R_{2}$ denoting edges of either color, and let $O_{1}, \ldots, O_{n} \subseteq O$ be connected subsets so that there does not exist an $R_{1} \vee R_{2}$-path of length at most $2^{n}$ between a vertex of $O_{i}$ and a vertex of $O_{j}$ for any distinct $i, j \leq n$. Let $\rho^{i}_{1}, \rho^{i}_{2} \subseteq O_{i}$ be a (C4)-coloring of $O_{i}$ for $i \leq n$. Then there is some $O' \leq O$ containing $O_{1} \cup \ldots \cup O_{n}$ and (C4)-coloring $\rho_{1}, 
\rho_{2}$ of $O$ so that, for $i \leq n$, $\rho_{1}, \rho_{2}$ restricts to $\rho^{i}_{1}, \rho^{i}_{2}$ on $O_{i}$.
\end{claim}

\begin{proof}
    Considering each connected component of $O$ individually, we may assume $O$ to be connected. We prove this claim by induction on $n$. Because we will later consider a variant of this construction where the case $n=2$ is the main difference, we isolate this case, which is necessary for the induction, as a subclaim:

    \begin{subclaim}
        \label{pathcoloring} Claim \ref{coloringinterpolation} is true where $O$ is connected and $n=2$.
    \end{subclaim}
\begin{proof}
    Let $I$ be the shortest path between $O_{1}$ to $O_{2}$, which will consist, ordered in the direction from $O_{1}$ to $O_{2}$, of  $o_{0}, \ldots o_{n}$, for $o_{0} \in O_{1}$, $o_{n} \in O_{2}$, $o_{1} \ldots o_{n-1} \in O \backslash (O_{1} \cup O_{2}) $, and $n \geq 5$. Because $O_{1}$ and $O_{2}$ are connected, it suffices to color $O_{1} \cup O_{2} \cup I$ so as to extend the colorings $\rho^{i}_{1}, \rho^{i}_{2}$ on $O_{i}$ for $i = 1, 2$, and to preserve the condition of being a (C4)-coloring. So color $O_{1}$ by $\rho^{1}_{1}, \rho^{1}_{2}$, $O_{2}$ by $\rho^{2}_{1}, \rho^{2}_{2}$, and color $o_{1} \ldots o_{n-1}$ as follows: first, color $o_{1}$ by $\rho_{3-i}$, for $i$ so that $o_{0}$ is an $R_{i}$-neighbor of $o_{1}$, and color $o_{n-1}$ by $\rho_{3-i}$, so that $o_{n-1}$ is an $R_{i}$-neighbor of $o_{n}$.  Since we colored $O_{1}$ by $\rho^{1}_{1}, \rho^{1}_{2}$ and $O_{2}$ by $\rho^{2}_{1}, \rho^{2}_{2}$, the only vertices $o$ of the $O_{i}$ so that the condition of being a (C4)-coloring can fail in $O_{1} \cup O_{2} \cup I$ on the boundary of the $R_{i}$-ball of radius $1$ centered at $o$ are $o_{0}$ and $o_{n}$, and we have just prevented this. It remains to color $o_{2} , \ldots o_{n-2}$; we color them all by $\rho_{2}$. The only remaining vertices $o$ of $O_{1} \cup O_{2} \cup I$ where the condition of being a (C4)-coloring can fail on the $1$-ball centered at $O$ are now $o_{1} \ldots o_{n-1}$, but the boundaries of the $1$-balls centered at those vertices in $O_{1} \cup O_{2} \cup I$ have just two points, at least one of which must be colored by $\rho_{2}$, because $n \geq 5$ and the interval $\{o_{2} , \ldots o_{n-2}\}$ whose points we colored by $\rho_{2}$ contains at least one of the two neighbors of each $o_{1}, \ldots, o_{n-1}$. So the condition of being a (C4)-coloring cannot fail there. 
\end{proof}
 We now consider the general case of the induction.  Without loss of generality, we may assume that $O_{n-1}$ and $O_{n}$ have minimum distance $d > 2^{n}$ among any $O_{i}, O_{j}$ with $i, j$ distinct. By Lemma \ref{pathcoloring}, we may find a (C4)-coloring $(\rho')^{n-1}_{1}, (\rho')^{n-1}_{2}$ of $O'_{n-1}=: O_{n-1} \cup O_{n} \cup I $ extending $\rho^{n-1}_{1}, \rho^{n-1}_{2}$ and $\rho^{n}_{1}, \rho^{n}_{2}$, where $I$ is the shortest path, which will be of length $d$, between $O_{n-1}$ and $O_{n}$. Clearly $O'_{n-1}$ is connected, so it suffices to show that no two of $O_{1}, \ldots O_{n-2}, O'_{n-1}$ have distance less than $\frac{d}{2} > 2^{n-1}$, because then we can apply the inductive step. Let $ i \leq n-2$; it suffices to show that $O_{i}$ does not have distance less than $\frac{d}{2}$ from  $O'_{n-1}$. Suppose otherwise; then it has distance less than $\frac{d}{2}$ from some point $p$ of $O'_{n-1}$, but by definition of $O'_{n-1}$, $p$ must have distance at most $\frac{d}{2}$ from either $O_{n-1}$ or $O_{n}$. So $O_{i}$ must have distance less than $\frac{d}{2} + \frac{d}{2} = d$ from either $O_{n-1}$ or $O_{n}$, contradicting minimality of $d$.
\end{proof}

By Claims \ref{coloringinterpolation} and \ref{coloringextension},  there is a (C4)-coloring $\rho_{1}, \rho_{2}$ on $O(A)$ extending the (C4)-coloring $\rho^{i}_{1}, \rho^{i}_{2}$ on $O_{i}$ for each $i \leq n$. So we can extend $A$ to $A'=A \cup p$, where $p \in P(A'\backslash A)$, and define the relations $\rho_{1}$ and $\rho_{2}$ on $O(A) \times \{p\}$ so that $p$ induces the (C4)-coloring $\rho_{1}, \rho_{2}$ on $O(A)$. This proves (***), and the consistency of $T^{\not\exists}$.

We call a copacetic $\mathcal{L}$-structure \textit{complete} if it satisfies axiom $T_{2}$; completeness of $A$ will always supplant condition (i) of $A \leq B$.

We now prove that $T^{\not\exists}$ has the following embedding property:

\begin{lemma}
  \label{embedding}  Let $\mathbb{M}$ be a sufficiently saturated model of $T^{\not\exists}$. Let $C \leq \mathbb{M}$, and let $B \geq C$ be a small copacetic $\mathcal{L}$-structure. Assume additionally that $B$ is complete. Then there is an embedding $\iota: B \hookrightarrow \mathbb{M}$ that is the identity on $C$ and satisfies $\iota(B) \leq \mathbb{M}$.
\end{lemma}
\begin{proof}
    We prove the lemma in the following two cases:

    (1) $P(B) = P(C)$

    (2) $O(B) = O(C)$ and $|P(B \backslash C)|=1$.

    These cases suffice, because the closed extension $C \leq B$ can be decomposed into a closed extension satisfying (1) followed by an ascending chain of (obviously closed) extensions satisfying (2), and the property $\iota(B) \leq \mathbb{M}$ is clearly preserved under taking unions.

    To prove case (1), by completeness of $B$ and saturatedness of $\mathbb{M}$, it suffices to find, for arbitrarily large $n < \omega$, an embedding $\iota_{n}: B \hookrightarrow \mathbb{M}$ that is the identity on $C$ and so that points of $\iota_{n}(B)$ that are not connected by a path in $\iota_{n}(B)$ have distance at least $n$ in $\mathbb{M}$. We claim that, for every copacetic $B' \geq C$, and any points $b, b' \in O(B)$ so that $b$ and $b'$ belong to distinct connected components of $B$ and $b'$ belongs to a connected component of $B$ not containing any point of $C$, there is some copacetic $B'' \supseteq B'$ so that $C \leq B''$, that consists of $B'$ together with a path of from $b$ to $b'$ of length greater than $n$. One way to do this is to add an $R_{1}$-path of length greater than $n$ (and greater than $2$) between $b$ and $b'$, so that all induced (C4)-colorings color the new nodes of this path by $\rho_{2}$.    By choice of induced colorings, the resulting $\mathcal{L}$-structure $B''$ will satisfy (C4) and condition (i) of $C \leq B$, and $B''$ will also satisfy condition (ii) of $C \leq B$ by the assumption of what connected components of $B''$ have $b$ and $b'$ as members. By repeatedly applying this claim to connect each connected component of $B$ not meeting $C$ to some fixed connected component of $B$ meeting $C$ (or an arbitrary fixed connected component of $B$, if $C$ is empty), we then obtain a copacetic $\mathcal{L}$-structure $B' \supseteq B$ with $P(B') = P(C)$ and $B' \geq C$ such that, for any $b, b' \in B$ that are not connected in $B$, either $b$ and $b'$ have finite distance greater than $n$ in $B'$, or they belong to different connected components of $B'$ meeting $C$. Now suppose $\iota': B' \hookrightarrow \mathbb{M}$ is any embedding restricting to the identity on $C$; then for $\iota'(b), \iota'(b') \in \iota'(B)$, that are not connected in $\iota'(B)$, either $b$ and $b'$ have finite distance greater than $n$ in $B'$, so $\iota(b)$ and $\iota(b')$ have distance greater than $n$ in $\mathbb{M}$ as desired, or $b$ and $b'$ belong to different connected components of $B'$ meeting $C$, so $\iota(b)$ and $\iota(b')$ belong to different connected components of $\iota(B')$ meeting $C$, so are not connected in $\mathbb{M}$ because $C \leq \mathbb{M}$; then $\iota_{n} = \iota'|_{B}$ is as desired.  So if we can show that, for any small, copacetic $\mathcal{L}$-structure $B'\geq C$ with $P(B') = P(C)$, there is an embedding $\iota : B' \to \mathbb{M}$ that is the identity on $C$ (with no additional requirements on $\iota(B)$), we will have proven case (1). Note that if $B_{0} \subseteq B'$, $B_{0} \cap C \leq B_{0} $, so we may assume that $B'$ is finite by saturatedness of $\mathbb{M}$. Then we can just apply tree extension.

    To prove case (2), let $\rho_{1}, \rho_{2}$ be the (C4)-coloring induced by $p$, where $\{p\} = P(B\backslash C)$, on $C$. Since $B$ is complete and $C \leq \mathbb{M}$, any embedding $\iota: B \to \mathbb{M}$ with $\iota|_{C} = \mathrm{id}$ will satisfy $\iota(B) \leq \mathbb{M}$. By saturatedness of $\mathbb{M}$ it suffices to find, for every finite $P_{0} \subset P(C)$ and finite $C_{0} \subset O(C)$,  some $p \in \mathbb{M} \backslash P_{0}$ inducing the (C4)-coloring $\rho_{1}|_{C_{0}}, \rho_{2}|_{C_{0}}$ on $C_{0}$. We may enlarge $C_{0}$ to a finite subset of $O(C)$ consisting of a finite union of $n$ connected sets that are not connected in $C$, so by condition (ii) of $C \leq \mathbb{M}$, are not connected in $\mathbb{M}$; in particular, any two of them have distance greater than $2^{n}$ in $\mathbb{M}$. So by parameter introduction, we may find infinitely many $p \in P(\mathbb{M})$ inducing the (C4)-coloring $\rho_{1}|_{C_{0}}, \rho_{2}|_{C_{0}}$ on $C_{0}$, so in particular, one not belonging to $P_{0}$.

\end{proof}
We now fix a sufficiently saturated $\mathbb{M} \models T^{\not\exists}$, which we take as the ambient model; for $p \in P(\mathbb{M})$, fix the notation, $p^{j}_{\to^{i}} =:p^{\mathbb{M}, j}_{\to^{i}}$   The following quantifier elimination is a corollary of the above:
\begin{cor}
\label{quantifierelimination} Let $A, B \leq \mathbb{M}$. Then if $\mathrm{qftp}_{\mathcal{L}}(A)=\mathrm{qftp}_{\mathcal{L}}(B)$, $\mathrm{tp}(A) = \mathrm{tp}(B)$.

\end{cor}
\begin{proof}
    For any $A \leq \mathbb{M}$ and $a \in \mathbb{M}$, there is some (small) $B \leq \mathbb{M}$ so that $Aa \subset B$, and in particular $A \leq B$. Because $B \leq \mathbb{M}$ and $\mathbb{M}$ satisfies completeness, $B$ is complete. So the corollary follows by Lemma \ref{embedding} by a back-and-forth argument.
\end{proof}

Next, we show that $T^{\not\exists}$ is $\mathrm{NSOP}_{1}$. By Theorem 9.1 of \cite{KR17}, it suffices to find an invariant ternary relation $\ind$ between subsets of $\mathbb{M}$ over models with the following properties:

(a) Strong finite character: For all $a, b$ and $M \models T$, if $a \nind_{M} b$, then there is some formula $\varphi(x, b) \in \mathrm{tp}(a/Mb)$, where $\varphi(x, y)$ has parameters in $M$, such that for every $a' \models \varphi(x, b)$, $a' \nind_{M} b$.

(b) Existence over models: for all $a$ and $M \models T$,  $a \ind_{M} M$.

(c) Monotonicity: For all $A' \subseteq A$, $B' \subseteq B$, $M \models T$, $A \ind_{M} B$ implies $A' \ind_{M} B'$.

(d) Symmetry: For all $a, b$ and $M \models T$, $a \ind_{M} b$ implies $b \ind_{M} a$ (and vice versa).

(e) The independence theorem: for any $a, a', b, c$ and $M \models T$,  $a \ind_{M} b$, $a' \ind_{M} c$, $b \ind_{M} c$ and $a \equiv_{M} a'$ implies that there is $a''$ with $a''b \equiv_{M} ab$, $a''c \equiv_{M} a'c$ and $a'' \ind_{M} bc$. 

First, note that for any $A$, there is $A' \supseteq A$ such that $A' \leq \mathbb{M}$ and for all $A''  \supseteq A$ with $A'' \leq \mathbb{M}$, $A' \subseteq A''$. Denote this $\mathrm{cl}(A)$; clearly, this is contained in $\mathrm{acl}(A)$ (and it can even be checked that $\mathrm{cl}(A) = \mathrm{acl}(A)$). This allows us to define $\ind$: $A \ind_{M} B$ if $\mathrm{cl}(MA) \cap \mathrm{cl}(MB) = M$, and every ($R_{1}(\mathbb{M}) \cup R_{2}(\mathbb{M})$)-path between a point of $O(\mathrm{cl}(MA) \backslash M)$ and $O(\mathrm{cl}(MB) \backslash M)$ contains a point of $M$. 

We see this has strong finite character: for $a \nind_{M} b$,  $a_{0} \in \mathrm{cl}(aM)$, $b_{0} \in \mathrm{cl}(bM)$ the endpoints of the path of length $n$ witnessing this, let $\varphi_{1}(y, b)$ isolate $b_{0}$ over $Mb$, $\varphi_{2}(x)$ say that $x$ is not within distance $n$ of the closest point of $M$ to some (any) $b'_{0} \models \varphi_{1}(x, b)$ (which would be required, should there be a (non-self-overlapping) path of length $n$ between $x$ and $b_{0}$ going through $M$), and $\varphi_{3}(x, b)$ say that there is a path of length $n$ from $x$ to some point $c$ satisfying $\models \varphi_{1}(c, b)$. Then $\varphi(x, b)=: \varphi_{2}(x) \wedge \varphi_{3}(x, b)$ suffices.

Existence over models, monotonicity, and symmetry are immediate. So it remains to show the independence theorem. By definition of $\ind$, we may assume $a=\mathrm{cl}(aM)$ and similar for $a', b, c$. We first give an analysis of the structure of pairs $a, b $ with $a \ind_{M} b$:

\begin{lemma}
  \label{pairstructure}  Let $a = \mathrm{cl}(aM)$, $b = \mathrm{cl}(bM)$, and $a \ind_{M} b$. Then $\mathrm{cl}(ab) = a^{*}b^{*}$, where for $P_{0}= P(ab)$, $a^{*}$ is the closure of $a$ under the functions $p^{j}_{\to^{i}}$ for $ p \in P_{0}$ (so in particular, $P(a^{*})=P(a)$), and $b^{*}$ is the closure of $b$ under these same functions. Every point of $O(a^{*} \backslash a)$ is connected by some ($R_{1} \vee R_{2}$)-path to a point of $a \backslash M$, and $a \leq a^{*}$; similarly for $b^{*}$.
\end{lemma}
\begin{proof}
The second sentence follows from the definition of the $p^{j}_{\to^{i}}$, and the fact that $M, a, b \leq \mathbb{M}$; it remains to show that $a^{*}b^{*} \leq \mathbb{M}$. But $a^{*}b^{*}$ is the closure of $ab$ under the functions $p^{j}_{\to^{i}}$ for $ p \in P_{0}$, so is complete, and condition (i) is satisfied. Moreover, $O(a)O(b) \leq \mathbb{M}$ (i.e., condition (ii) of being closed in $\mathbb{M}$ holds for $ab$): any path with endpoints in $O(a)O(b)$ and intermediate points all in $O(\mathbb{M} \backslash ab)$ must, because $a \ind_{M} b$ and $M \subseteq a, b$, have both endpoints in either of $O(a)$ or $O(b)$, a contradiction because $O(a)$ and $O(b)$ are closed in $\mathbb{M}$. Because every point of $O(a^{*})O(b^{*})$  is connected to $ab$ by the second clause of this lemma, condition (ii) of $a^{*}b^{*} \leq \mathbb{M}$ then follows from $O(a)O(b) \leq \mathbb{M}$.
\end{proof}

In the following, for $P_{0} \subset P(\mathbb{M})$ and $A \subset \mathbb{M}$, let $\mathrm{cl}_{P_{0}}(A)$ denote the closure of $A$ under $p^{j}_{\to^{i}}$ for $p \in P_{0}$, as in Lemma \ref{pairstructure}. We want to find $\tilde{a}, \tilde{b}, \tilde{c}$ so that $\tilde{a}\tilde{b} \equiv_{M} ab$, $\tilde{a}\tilde{c} \equiv_{M} a'c$, $\tilde{b}\tilde{c} \equiv_{M} bc$, $\tilde{a} \ind_{M} \tilde{b}\tilde{c}$. Let $b^{*_{c}}= \mathrm{cl}_{P(b)P(c)}(b)$, $c^{*_{b}}= \mathrm{cl}_{P(b)P(c)}(c)$, $b^{*_{a}}= \mathrm{cl}_{P(a)P(b)}(b)$, $c^{*_{a'}}= \mathrm{cl}_{P(a')P(c)}(c)$, $a^{*_{b}}= \mathrm{cl}_{P(a)P(b)}(a)$, $a^{*_{c}}= \mathrm{cl}_{P(a')P(c)}(a')$. We now build a copacetic $\mathcal{L}$-structure extending $M$ as follows. Let $A_{0} =: \tilde{a}_{0}\tilde{b}_{0}\tilde{c}_{0}$ where $R_{1}, R_{2}$ is defined on $O(\tilde{a}_{0}), O(\tilde{b}_{0}), O(\tilde{c}_{0})$ so that these are freely amalgamated sets isomorphic to $O(a), O(b), O(c)$ over $O(M)$ (so in particular, $O(a) O(b ) \equiv^{\mathcal{L}-\mathrm{qftp}}_{O(M)} O(\tilde{a}_{0})O(\tilde{b}_{0})$ because $a \ind_{M} b$, and so on), and so that $P(\tilde{a}_{0}), P(\tilde{b}_{0}), P(\tilde{c}_{0})$ and $\rho_{1}, \rho_{2}$ are defined so that $\tilde{a}_{0}\tilde{b}_{0} \equiv^{\mathcal{L}-\mathrm{qftp}}_{M} ab$, $\tilde{a}_{0}\tilde{c}_{0} \equiv^{\mathcal{L}-\mathrm{qftp}}_{M} a'c$, $\tilde{b}_{0}\tilde{c}_{0} \equiv^{\mathcal{L}-\mathrm{qftp}}_{M} bc$; this makes sense because $a \equiv_{M} a'$. Then $A_{0}$ is copacetic, because the only way axiom (C4) can fail is, without loss of generality, when there is $p \in P(\tilde{a}_{0} \backslash M)$, $o_{1} \in O(\tilde{b}_{0} \backslash M)$, $o_{2} \in O(\tilde{c}_{0} \backslash M)$ so that $o_{1}, o_{2}$ lie on the boundary of some common $R_{1}$-ball and $A_{0} \models \rho_{1}(p, o_{1}) \wedge \rho_{1}(p, o_{2})$. Because $O(\tilde{b}_{0}), O(\tilde{c}_{0})$ are freely amalgamated over $M$, the only way the former can happen is for $o_{1}$ and $o_{2}$ to have a common $R_{1}$-neighbor $m \in M$. But then (say) $o_{1} = p^{1}_{\to^{1}}(m) \in \tilde{a}_{0}$, contradicting $\tilde{a}_{0} \cap \tilde{b}_{0} = M$.  Now extend $A_{0}$ to $A_{1}=: \tilde{a}^{*_{\tilde{b}}}\tilde{a}^{*_{\tilde{c}}}\tilde{b}^{*_{\tilde{a}}}\tilde{b}^{*_{\tilde{c}}}\tilde{c}^{*_{\tilde{a}}}\tilde{c}^{*_{\tilde{b}}} $, where 

(1)$\tilde{a}^{*_{\tilde{b}}}= \tilde{a}_{0} \cup O(\tilde{a}^{*_{\tilde{b}}} \backslash \tilde{a}_{0})$, $\tilde{a}^{*_{\tilde{c}}}= \tilde{a}_{0} \cup O(\tilde{a}^{*_{\tilde{c}}} \backslash \tilde{a}_{0})$, $\tilde{b}^{*_{\tilde{a}}}= \tilde{b}_{0} \cup O(\tilde{b}^{*_{\tilde{a}}} \backslash \tilde{b}_{0})$, $\tilde{b}^{*_{\tilde{c}}}= \tilde{b}_{0} \cup O(\tilde{b}^{*_{\tilde{c}}} \backslash \tilde{b}_{0})$, $\tilde{c}^{*_{\tilde{a}}}= \tilde{c}_{0} \cup O(\tilde{c}^{*_{\tilde{a}}} \backslash \tilde{c}_{0})$, $\tilde{c}^{*_{\tilde{b}}}= \tilde{c}_{0} \cup O(\tilde{c}^{*_{\tilde{b}}} \backslash \tilde{c}_{0})$, 

(2) $O(\tilde{a}^{*_{\tilde{b}}} \backslash \tilde{a}_{0})$, $O(\tilde{a}^{*_{\tilde{c}}} \backslash \tilde{a}_{0})$, $O(\tilde{b}^{*_{\tilde{a}}} \backslash \tilde{b}_{0})$, $O(\tilde{b}^{*_{\tilde{c}}} \backslash \tilde{b}_{0})$, $O(\tilde{c}^{*_{\tilde{a}}} \backslash \tilde{c}_{0})$, $O(\tilde{c}^{*_{\tilde{b}}} \backslash \tilde{c}_{0})$ are pairwise disjoint and disjoint from $A_{0}$

(3) The only $R_{i}$-edges of $A_{1}$ are those of $A_{0}$, as well as those required to make $O(\tilde{a}^{*_{\tilde{b}}}) \equiv^{\mathcal{L}-\mathrm{qftp}}_{O(M)} O(a^{*_{b}})$,  $O(\tilde{a}^{*_{\tilde{c}}}) \equiv^{\mathcal{L}-\mathrm{qftp}}_{O(M)} O(a'^{*_{c}})$, $O(\tilde{b}^{*_{\tilde{a}}}) \equiv^{\mathcal{L}-\mathrm{qftp}}_{O(M)} O(b^{*_{a}})$, $O(\tilde{b}^{*_{\tilde{c}}}) \equiv^{\mathcal{L}-\mathrm{qftp}}_{O(M)} O(b^{*_{c}})$, $O(\tilde{c}^{*_{\tilde{a}}}) \equiv^{\mathcal{L}-\mathrm{qftp}}_{O(M)} O(c^{*_{a'}})$, $O(\tilde{c}^{*_{\tilde{b}}}) \equiv^{\mathcal{L}-\mathrm{qftp}}_{O(M)} O(c^{*_{b}})$.

(4) Let us extend $\rho_{1}\rho_{2}$ where required so that $\tilde{a}^{*_{\tilde{b}}} \tilde{b}^{*_{\tilde{a}}} \equiv^{\mathcal{L}-\mathrm{qftp}}_{M} a^{*_{b}} b^{*_{a}}$, $\tilde{a}^{*_{\tilde{c}}} \tilde{c}^{*_{\tilde{a}}} \equiv^{\mathcal{L}-\mathrm{qftp}}_{M} a'^{*_{c}} c^{*_{a'}}$, $\tilde{b}^{*_{\tilde{c}}} \tilde{c}^{*_{\tilde{b}}} \equiv^{\mathcal{L}-\mathrm{qftp}}_{M} b^{*_{c}} c^{*_{b}}$. Because $\tilde{a}^{*_{\tilde{b}}} \tilde{b}^{*_{\tilde{a}}}, \tilde{a}^{*_{\tilde{c}}} \tilde{c}^{*_{\tilde{a}}}, \tilde{b}^{*_{\tilde{c}}} \tilde{c}^{*_{\tilde{b}}} $ are already known to satisfy (C4), a failure of (C4), which would then be witnessed by $p \in P(A_{1})$, $o_{1}, o_{2}, o_{3} \in O(A_{1})$ (with the last two perhaps equal), could now happen in the following two cases (among the instances where the $\rho_{i}$ are yet defined), both of which we rule out. First, one of the $o_{i}$ is in $a^{\dagger} \backslash{A_{0}}$, where $a^{\dagger}$ is one of $\tilde{a}^{*_{\tilde{b}}}\tilde{a}^{*_{\tilde{c}}}, \tilde{b}^{*_{\tilde{a}}}\tilde{b}^{*_{\tilde{c}}}, \tilde{c}^{*_{\tilde{a}}}\tilde{c}^{*_{\tilde{b}}}$, and the other is in $A_{1} \backslash a^{\dagger} $. But then, by the connectedness claim in the second clause of Lemma \ref{pairstructure}, (3) and the construction of $O(A_{0})$ tell us that these two $o_{i}$ must have distance at least $2$ apart, so a failure of (C4) cannot arise here. The other case is where $o_{1}, o_{2}, o_{3} \in a^{\dagger} $ for $a^{\dagger}$ one of $\tilde{a}^{*_{\tilde{b}}}\tilde{a}^{*_{\tilde{c}}}, \tilde{b}^{*_{\tilde{a}}}\tilde{b}^{*_{\tilde{c}}}, \tilde{c}^{*_{\tilde{a}}}\tilde{c}^{*_{\tilde{b}}}$, and some two do not belong to the same $\tilde{a}^{*_{\tilde{b}}},\tilde{a}^{*_{\tilde{c}}},\tilde{b}^{*_{\tilde{a}}},\tilde{b}^{*_{\tilde{c}}},\tilde{c}^{*_{\tilde{a}}},\tilde{c}^{*_{\tilde{b}}}$; then $p \in a^{\dagger}$ because this is the only way the $\rho_{i}$ can be defined so far for $p$, $o_{1}, o_{2}, o_{3}$. Then these two $o_{j}$, say $o_{1}$ and $o_{2}$ must satisfy $\rho_{i}(p, o_{1}), \rho_{i}(p, o_{2}) $ for some $i \in \{1, 2\}$, and must be common $R_{i}$-neighbors of the evident $\tilde{a}, \tilde{b}, \tilde{c}$ while lying outside of this $\tilde{a}, \tilde{b}, \tilde{c}$. But this is impossible, by the claim $a \leq a^{*}$, $b \leq b^{*}$ of Lemma \ref{pairstructure}. Therefore, we do not yet get a failure of (C4), and it remains to extend $\rho_{1}, \rho_{2}$ to get (C3) while maintaining (C4), which we do in the next step.

(5) By the claim $a \leq a^{*}$, $b \leq b^{*}$ in the second clause of Lemma \ref{pairstructure}, we can use Claim \ref{coloringextension} to extend $\rho_{1}, \rho_{2}$ where not yet defined, maintaining (C4) by the description of the $R_{1}, R_{2}$-structure in (3), and thereby producing a copacetic $\mathcal{L}$-structure.

Observe also the following:

(6) In $A_{0}$, there is no path between $\tilde{a}$ and $\tilde{b}\tilde{c}$ not going through $M$. So by the connectedness claim in Lemma \ref{pairstructure} applied to $c^{*_{b}}, b^{*_{c}}$, in $A_{1}$ there is no path between $\tilde{a}$ and $\tilde{c}^{*_{\tilde{b}}} \tilde{b}^{*_{\tilde{c}}}$ not going through $M$. 

(7). By construction, $O(\tilde{a}), O(\tilde{b}), O(\tilde{c})$ are closed in $O(A_{1})$. So by the proof of the first clause of Lemma \ref{pairstructure}, $\tilde{a}^{*_{\tilde{b}}} \tilde{b}^{*_{\tilde{a}}}, \tilde{a}^{*_{\tilde{c}}} \tilde{c}^{*_{\tilde{a}}}, \tilde{b}^{*_{\tilde{c}}} \tilde{c}^{*_{\tilde{b}}}$ are each closed in $A_{1}$ (and are also each complete).

Finally, we extend $A_{1}$ to a \textit{complete} copacetic $\mathcal{L}$-structure, so that (6) and (7) still hold replacing $A_{1}$ with $A_{2}$; since $\tilde{a}^{*_{\tilde{b}}} \tilde{b}^{*_{\tilde{a}}}, \tilde{a}^{*_{\tilde{c}}} \tilde{c}^{*_{\tilde{a}}}, \tilde{b}^{*_{\tilde{c}}} \tilde{c}^{*_{\tilde{b}}}$  are complete, to preserve (7) we must just preserve clause (ii). We can extend $A_{1}$ to a complete copacetic $\mathcal{L}$-structure just by repeated applications of the proof of (ii). But notice that this proceeds just by successively adding nodes in the sort $O$ with exactly one $R_{1} \vee R_{2}$-neighbor in the previously added nodes, so adds no new paths between nodes in $O(A_{1})$. So (6) and (7) are in fact preserved.

Note that $A \leq B \leq C$ implies $A \leq C$, so by (4) and (7) (i.e., the version where $A_{1}$ is replaced with $A_{2}$), $M \leq A_{2}$.  Now use Lemma \ref{embedding} to obtain an embedding $\iota: A_{2} \hookrightarrow \mathbb{M}$ which is the identity on $M$, and such that $\iota(A_{2}) \leq  \mathbb{M}$. Let $\tilde{a}=\iota(\tilde{a}_{0})$,  $\tilde{b}=\iota(\tilde{b}_{0})$,  $\tilde{c}=\iota(\tilde{c}_{0})$. Then (using Corollary \ref{quantifierelimination}) by (4), (7) (again, the version where $A_{1}$ is replaced with $A_{2}$), $\iota(A_{2}) \leq M$, and the first clause of Lemma \ref{pairstructure}, $\tilde{a}\tilde{b} \equiv_{M} ab$, $\tilde{a}\tilde{c} \equiv_{M} a'c$, $\tilde{b}\tilde{c} \equiv_{M} bc$. Moreover, by (6) (yet again, the version where $A_{1}$ is replaced with $A_{2}$) and $\iota(A_{2}) \leq M$ , $\tilde{a} \ind_{M} \tilde{b}\tilde{c}$. So we have proven the independence theorem for $\ind$, and $T^{\not\exists}$ is $\mathrm{NSOP}_{1}$.

It remains to show that $T^{\not\exists}$ does not satisfy the existence axiom. Let $p(x)$ be the unique type in the sort $P$ (in one variable) over $\emptyset$. Let $o \in O(\mathbb{M})$. Then $p(x) \vdash \rho_{1}(x, o) \vee \rho_{2}(x, o) $, by (C3). We show that $\rho_{1}(x, o)$ $2$-divides over $\emptyset$. We can find an $\emptyset$-indiscernible sequence $\{o_{i}\}_{i < \omega}$, $o_{0} = o$, so that the $o_{i}$ lie on the boundary of some fixed $R_{1}$-ball of radius $1$. Then $\{\rho(x, o_{i})\}_{i < \omega}$ is $2$-inconsistent by (C4), so $\rho_{1}(x, o_{i})$ $2$-divides over $\emptyset$.  That $\rho_{2}(x, o_{i})$ $3$-divides over $\emptyset$ will be similar. So $p \in S(\emptyset)$ forks over $\emptyset$, violating the existence axiom.

This proves the main theorem of this paper, Theorem \ref{main}, and answers the main question, Question \ref{mainquestion}.

\begin{remark}\label{banff}
    It is not too hard to show that $\ind$ above even satisfies the following axiom:

    (f) Witnessing: Let $a  \nind_{M} b$, and let $\{b_{i}\}_{i < \omega}$, $b_{0} = b$, be an $M$-indiscernible sequence with $b_{i} \ind_{M} b_{0} \ldots b_{i-1}$ for all $i < \omega$. Then there is a formula $\varphi(x, b) \in \mathrm{tp}(a/Mb)$, $\varphi(x, y) \in L(M)$, so that $\{\varphi(x, b_{i})\}_{i < \omega}$ is inconsistent.

    Theorem 6.11 of \cite{KR19} says that if an invariant ternary relation $\ind$ between subsets of $\mathbb{M}$ over models satisfies strong finite character, existence over models, monotonicity, symmetry, the independence theorem and witnessing, then $\ind$ coincides with Kim-independence $\ind^{K}$ (Definition \ref{kimindependence}.) So in $T^{\not\exists}$, Kim-independence (over models) is given by $\ind$. Because $T^{\not\exists}$ does not satisfy the existence axiom, the results on Kim-independence over \textit{sets} from \cite{DKR22}, \cite{CKR20}) do not apply to $T^{\not\exists}$. But note that it makes sense to define $a \ind_{C} b$ the same way as above when ($\mathrm{acl}(C) =$) $ \mathrm{cl}(C) = C$, and, when $C$ is any set, define $a \ind_{C} b$ by $a \ind_{\mathrm{cl}(C)} b$, giving a ternary relation on sets. Over sets, by the same proofs as above, $\ind$ satisfies the analogues of strong finite character, monotonicity, symmetry, the independence theorem (where $\equiv_{C}$ is replaces by $\equiv^{\mathrm{Lstp}}_{C}$, though this is the same as $\equiv_{\mathrm{cl}(C)}$, with respect to which the independence theorem holds for $\ind$), and witnessing; moreover $\ind$ satisfies a stronger version of existence over sets:

    (b$'$) Existence and extension over sets: for any $a, C$ and $B' \supseteq B \supseteq C$, $a \ind_{C} C$, and if $a \ind_{C} B$ there is $a' \equiv_{B} a$ with $a' \ind_{C} B'$. 

    Ramsey, in a presentation at the Banff International Research Station on joint work with Itay Kaplan (\cite{Rbanff2023}), defines the assertion that \textit{Kim-independence is defined over sets} to mean that there is a ternary relation $\ind$ between sets satisfying the analogues over sets of  strong finite character, monotonicity, symmetry, the independence theorem, and witnessing, as well as existence and extension over sets, and shows that such $\ind$ is uniquely determined when it exists. So by \cite{DKR22}, \cite{CKR20}, in any $\mathrm{NSOP}_{1}$ theory satisfying the existence axiom, Kim-independence is defined over sets in the sense of \cite{Rbanff2023}. But also, despite $T^{\not\exists}$ not satisfying the existence axiom, Kim-independence is defined over sets in this sense in the theory $T^{\not\exists}$, even if the results of \cite{DKR22} on Kim-independence as defined by Dobrowolski, Kim, and Ramsey (Definition \ref{setkimindependence} above) do not apply in $T^{\not \exists}$.\footnote{In fact, in $T^{\not \exists}$, $\ind$ actually coincides with $\ind^{K}$ as defined by \cite{DKR22} (i.e. Kim-forking independence with respect to nonforking Morley sequences, Definition \ref{setkimindependence} above), so the \textit{conclusions} of, say, Corollary 4.9 or Theorem 5.6 of \cite{DKR22} hold: $\ind^{K}$ as defined there is symmetric, and satisfies the independence theorem, over arbitrary sets. (If $a \nind_{C} b $ and $aC$ is algebraically closed, then $\mathrm{tp}(a/Cb)$ implies a finite disjunction of formulas of the form $\varphi(x, b')$, where $b' \notin \mathrm{cl}(C)$ is a \textit{singleton} of $O$ or $P$ and $\varphi(x, b')$ either says that $x = b'$ or implies that there is a path between $x$ and $b'$ with no points in $\mathrm{cl}(C)$, and a formula of either kind divides over $C$ with respect to a $\mathrm{cl}(C)$-invariant Morley sequence. On the other hand, $a \ind_{C} b$ implies $a \ind_{C} M$ for $M$ some $|C|^{+}$-saturated model containing $Cb$, so $a \ind^{K}_{C} M$ by the independence theorem and $|C|^{+}$-saturatedness of $M$; see the clause $\ind \Rightarrow \ind^{K}$ of Theorem 9.1 of \cite{KR17}, and the standard argument that forking-dependence on a sufficiently saturated model implies dividing-dependence on that model.) But Proposition 4.9 of \cite{DKR22} fails--it is not necessarily true that $\varphi(x, b)$ forks over $C$ with respect to nonforking Morley sequences if and only if it divides with respect to nonforking Morley sequences. For example, for $o \in O, p \in P$, let $\varphi(x, op) =: x = p$; then $\varphi(x, op)$ does not divide with respect to a nonforking Morley sequence over $\emptyset$ (i.e. Kim-divide over $\emptyset$, as in Definition \ref{setkimindependence}), because there \textit{are} no nonforking Morley sequences over $\emptyset$ starting with $op$), but it implies $\tilde{\varphi}(x, p) =: x = p$, which does divide with respect to a nonforking Morley sequence over $\emptyset$, so $\varphi(x, op)$ Kim-forks over $\emptyset$. Moreover, Kim's lemma, Theorem 3.5 of \cite{DKR22}, is also false in $T^{\not \exists}$:  $\varphi(x, op) $ divides over $\emptyset$ with respect to \textit{all} nonforking Morley sequences over $\emptyset$ starting with $op$, but not with respect to \textit{some} nonforking Morley sequence over $\emptyset$ starting with $op$!

    By way of obtaining an $\mathrm{NSOP}_{1}$ theory where $\ind^{K}$ as defined over sets by \cite{DKR22} (Defintion \ref{setkimindependence} here) does not, say, satisfy the independence theorem, we expect that, by an extremely tedious verification, $T^{\not\exists}$ can be shown to eliminate $\exists^{\infty}$. So by Theorem 5 of \cite{W75} and Theorem 4.5 of \cite{KR18}, the generic expansion of $T^{\not\exists}$ by functions from $P$ to $O$ and from $O$ to $p$ (i.e. the model companion of models of (the Morleyization of) $T^{\not\exists}$ expanded by a unary function from sort $P$ to sort $O$ and a unary function from sort $O$ to sort $P$) will exist and have $\mathrm{NSOP}_{1}$, and no consistent formula can Kim-divide over $\emptyset$, because every nonempty parameter will have an element of $O$ and an element of $P$ in its definable closure, so can be shown to begin no Morley sequence over $\emptyset$ as in the original proof that $T^{\not\exists}$ does not satisfy the existence axiom. So, using Definition \ref{setkimindependence} to define $\ind^{K}$ over arbitrary sets, any set will be Kim-independent over $\emptyset$ from any nonempty set.

    } So the results stated in \cite{Rbanff2023} are independent of the previous work on the existence axiom.
    \end{remark}

\section{Quantitative results}

Doborowolski, Kim, and Ramsey show, in Remark 6.7 of \cite{DKR22}, that in a theory without the strict order property (i.e. an $\mathrm{NSOP}$ theory), the failure of the existence axiom cannot be witnessed by two formulas that $2$-divide:

\begin{fact}\label{existencensop}
    Let $T$ be $\mathrm{NSOP}$, and $p \in S(A)$. Then there are no formulas $\varphi_{1}(x, b_{1})$, $\varphi_{2}(x, b_{2})$, each of which $2$-divide over $A$, such that $p \vdash \varphi_{1}(x, b_{1}) \vee \varphi_{2}(x, b_{2})$.
\end{fact}

In the previous section, we gave an example, $T^{\not\exists}$, of an $\mathrm{NSOP}_{1}$ theory where, for $p \in S(\emptyset)$, $p \vdash \varphi_{1}(x, b) \vee \varphi_{2}(x, b)$, where $\varphi_{1}(x, b)$ $2$-divides over $\emptyset$ and $\varphi_{2}(x, b)$ $3$-divides over $\emptyset$. Here, we describe an example, $T^{\not\exists^{2,2,2}}$, of an $\mathrm{NSOP}_{1}$ theory where, for $p \in S(\emptyset)$, $p \vdash \varphi_{1}(x, b) \vee \varphi_{2}(x, b) \vee \varphi_{3}(x, b)$, where for $i = 1, 2, 3$ each $\varphi_{i}(x, b)$ $2$-divides over $\emptyset$. This will show the optimality of Fact \ref{existencensop}.

Let $\mathcal{L}$ be the language with sorts $P$ and $O$, symbols $R_{1}$, $R_{2}$ and $R_{3}$ for binary relations on $O$, and symbols $\rho_{1}$, $\rho_{2}$, and $\rho_{3}$ for binary relations between $P$ and $O$. Call an $\mathcal{L}$-structure $A$ \textit{copacetic}$^{2,2,2}$ if:

(C1)$^{2,2,2}$ For $i = 1, 2, 3$, $R_{i}(A)$ is a symmetric, irreflexive relation on $O(A)$, and the three are mutually exclusive: for $a_{1}, a_{2} \in O(A)$, $A \not\models R_{i}(a_{1}, a_{2}) 
 \wedge R_{j}(a_{1}, a_{2})$ for $i \neq j \in \{1, 2, 3\}$

 (C2)$^{2,2,2}$ The relation $R_{1}(A) \cup R_{2}(A) \cup R_{3}(A)$ has no loops on $O(A)$ (i.e. there are no distinct $a_{0} \ldots a_{n-1} \in O(A)$, $n > 2$, and $i_{1} \ldots i_{n} \in \{1, 2, 3\}$ so that, for $0 \leq j \leq n-1$, $A \models R_{i_{j}}(a_{i}, a_{i+1 \mathrm{\: mod \:} n})$).

 (C3)$^{2,2,2}$ For all $b \in P(A)$, $a \in O(A)$, exactly one of $A \models \rho_{1}(b, a)$, $A \models \rho_{2}(b, a)$, and $\rho_{3}(b, a)$ hold.

 (C4)$^{2,2,2}$: For $i \in \{1, 2, 3\}$, there is no $b \in P(A)$ and distinct $a_{1}, a_{2}$ on the boundary of some fixed unit $R_{i}$-ball so that $A \models \rho_{i}(b, a_{1}) \wedge \rho_{i}(b, a_{2})$.

 We define the closure relation $\leq$ analogously to the previous section, and construct a theory satisfying the analogous statement to Lemma \ref{embedding}, which will be $\mathrm{NSOP}_{1}$ and satisfy $p \vdash \rho_{1}(x, o) \vee \rho_{2}(x, o) \vee \rho_{3}(x, o)$ for any $o \in O(\mathbb{M})$ and $p \in S(\emptyset)$ the unique type (in one variable) in sort $P$ over $\emptyset$; $\rho_{i}(x, o)$ will $2$-divide over $\emptyset$ for $i \in \{1, 2, 3\}$, as desired. The entire proof is a straightforward generalization of the previous section, with a single exception: in place of Subclaim \ref{pathcoloring}, we must prove the below subclaim. Let $O$ be an undirected graph without cycles and with a $3$-coloring of its edges, with $R_{1}$, $R_{2}$, $R_{3}$ denoting edges of each color. Let $\rho_{1}, \rho_{2}, \rho_{3} \subset O$, $O = \rho_{1} \cup \rho_{2} \cup \rho_{3}$, $\rho_{i} \cap \rho_{j} = \emptyset$ for $i \neq j \in \{1, 2, 3\}$  be a coloring of the vertices of $O$ so that, for $i = \{1, 2, 3\}$, no two distinct vertices of $O$, lying on the boundary of the same $R_{i}$-ball of radius $1$ (i.e. they have a common $R_{i}$-neighbor), are both colored by $\rho_{i}$. Then we call $\rho_{1}, \rho_{2}, \rho_{3}$ a \textit{(C4)$^{2,2,2}$-coloring} of $O$.

 \begin{subclaim}\label{alternatepathcoloring}
     Let $O$ be a connected graph without cycles, and with a $3$-coloring of its edges. Let $O_{1}, O_{2} \subset O$ be connected subgraphs so that each vertex of $O_{1}$ has distance at least $5$ from each vertex of $O_{2}$. For $i = 1, 2$, let $\rho_{1}^{i}, \rho_{2}^{i}, \rho_{3}^{i}$ be a (C4)$^{2,2,2}$-coloring of $O_{i}$. Then there is a (C4)$^{2,2,2}$-coloring $\rho_{1}, \rho_{2}, \rho_{3}$ of some connected set $O'$ containing $O_{1}$ and $O_{2}$, where for $i = 1, 2$, $\rho_{1}, \rho_{2}, \rho_{3}$ extends $\rho_{1}^{i}, \rho_{2}^{i}, \rho_{3}^{i}$ on $O_{i}$.
 \end{subclaim}

 \begin{proof}
     As in the proof of Subclaim \ref{pathcoloring}, let $O'= O_{1} \cup O_{2} \cup I$ where $I$ is the shortest path between $O_{1}$ and $O_{2}$, and let $I$ consist, ordered in the direction from $O_{1}$ to $O_{2}$, of $o_{0}, \ldots o_{n}$, for $o_{0} \in O_{1}$, $o_{n} \in O_{2}$, $o_{1} \ldots o_{n-1} \in O \backslash (O_{1} \cup O_{2}) $, and $n \geq 5$. Again, as in that proof, color $O_{i}$ by $\rho^{i}_{1}, \rho^{i}_{2}, \rho^{i}_{3}$ for $i \in \{1, 2\}$, color $o_{1}$ by $\rho_{i}$ where $i$ is such that $o_{1}$ is not an $R_{i}$-neighbor of $o_{0}$, and color $o_{n-1}$ by $\rho_{j}$ where $j$ is such that $o_{n-1}$ is not an $R_{j}$-neighbor of $o_{n}$--then as before, the condition of being a (C4)$^{2,2,2}$-coloring cannot fail at the boundary of a unit ball centered at a point of $O_{1}$ or $O_{2}$. Now let $n_{\mathrm{even}}, n_{\mathrm{odd}}$, respectively, be the least even and odd numbers less than $n-1$. Then, because there are three colors available, we can color $o_{2}, \ldots, o_{2i}, \ldots, o_{n_{\mathrm{even}}}$ so that each vertex in the sequence $o_{0}, o_{2}, \ldots, o_{2i}, \ldots, o_{n_{\mathrm{even}}}, o_{n_{\mathrm{even}}+2}$ is colored differently from the previous vertex in that sequence--noting that the colors of $o_{0}$ and $o_{n_{\mathrm{even}}+2}$ are already decided, alternate the color of $o_{0}$ with a color distinct from that of $o_{0}$ and $o_{n_{\mathrm{even}}+2}$. Similarly, we can color $o_{3}, \ldots, o_{2i+1} , \ldots o_{n_{\mathrm{odd}}}$ so that each vertex in the sequence $o_{1}, o_{3}, \ldots, o_{2i+1}, \ldots, o_{n_{\mathrm{odd}}}, o_{n_{\mathrm{odd}}+2}$ is colored differently from the previous vertex in that sequence. Coloring the intermediate vertices $o_{2}, \ldots o_{n-2}$ according to these observations, we see that the condition of being a (C4)$^{2,2,2}$-coloring cannot fail on the boundary of a unit ball centered at one of $o_{1}, \ldots o_{n-1}$, because the boundary of such a ball will always be colored by two different colors.
 \end{proof}

Note that a similar subclaim would fail, were we to try to use an analogous construction to obtain an $\mathrm{NSOP}_{1}$ theory where, for $p \in S(\emptyset)$, $p \vdash \varphi_{1}(x, b_{1}) \vee \varphi_{2}(x, b_{2})$ for $\varphi_{i}(x, b_{i})$ $2$-dividing over $\emptyset$.

\section{Open questions}

The theory $T^{\not \exists}$, despite being an $\mathrm{NSOP}_{1}$ theory that does not satisfy the existence axiom, is not countably categorical. Motivated by this, we ask:

\begin{question}
    Does every countably categorical $\mathrm{NSOP}_{1}$ (or even $\mathrm{NSOP}$) theory satisfy the existence axiom?
\end{question}

Moreover, in $T^{\not\exists}$, Kim-independence over models is not just given by the operation $\mathrm{acl}^{eq}$; see Remark \ref{banff}. In Definition 6.10 of \cite{KoponenConjecture}, the definition of the property of being one-based is extended (up to elimination of hyperimaginaries) from simple theories to $\mathrm{NSOP}_{1}$ theories:

\begin{definition}\label{onebased}
     Let $T$ be an $\mathrm{NSOP}_{1}$ theory. Then $T$ is \emph{one-based} if $A \nind^{K}_{M} B$ implies (equivalently, is equivalent to) $\mathrm{acl}^{eq}(AM) \cap \mathrm{acl}^{eq}(BM) \supsetneq M$.
\end{definition}

So $T^{\not\exists}$ is not one-based. (See Example 4.6.1 of \cite{P96}.) This leads us to ask:

\begin{question}
    Does every one-based $\mathrm{NSOP}_{1}$ theory satisfy the existence axiom?
\end{question}

Recall that, as stated in Remark \ref{banff}, Kim-independence is defined over sets in any $\mathrm{NSOP}_{1}$ theory satisfying the existence axiom, but is also defined over sets in $T^{\not\exists}$ despite $T^{\not\exists}$ not satisfying the existence axiom. A final question, motivated by this remark and by the original motivation discussed in the introduction for Question \ref{mainquestion}, the main question of this paper, is asked by Ramsey:

\begin{question}\label{banffquestion}(\cite{Rbanff2023}, \cite{Rbanff2024})
    Is Kim-independence defined over sets in every $\mathrm{NSOP}_{1}$ theory?
\end{question}

\section*{Acknowledgements} The author would like to thank James Freitag, Maryanthe Malliaris and Nicholas Ramsey for many insightful conversations. In particular, conversations with Nicholas Ramsey were instrumental in inspiring the discussion in Remark \ref{banff} and Question \ref{banffquestion} of this paper.

\bibliographystyle{plain}
\bibliography{refs}

\end{document}